\documentclass[a4paper]{amsart}

\usepackage[T1]{fontenc}
\usepackage[utf8]{inputenc}
\usepackage[english]{babel}

\usepackage[hidelinks]{hyperref}

\usepackage{booktabs}
\usepackage{enumitem}
\usepackage{amssymb}
\usepackage{amsmath}
\usepackage{amsthm}
\usepackage{mathtools}

\usepackage{lmodern}

\newcommand*\diff{\mathop{}\!\mathrm{d}} 

\newcommand{\pc}[1]{
  \begingroup\lccode`~=`: \lowercase{\endgroup
  \edef~}{\mathbin{\mathchar\the\mathcode`:}\nobreak}%
  (%
  \begingroup
  \mathcode`:=\string"8000
  #1%
  \endgroup
  )%
}

\newcommand{\cL}{\mathcal{L}}
\newcommand{\cO}{\mathcal{O}}

\newcommand{\mL}{\widetilde{\mathcal{L}}}
\newcommand{\mT}{\mathcal{T}}
\newcommand{\mU}{\mathcal{U}}
\newcommand{\mV}{\mathcal{V}}
\newcommand{\mX}{\mathcal{X}}

\newcommand{\bA}{\mathbb{A}}

\newcommand{\FF}{\mathbb{F}}
\newcommand{\GG}{\mathbb{G}}
\newcommand{\PP}{\mathbb{P}}
\newcommand{\QQ}{\mathbb{Q}}
\newcommand{\QQbar}{{\overline{\mathbb{Q}}}}
\newcommand{\RR}{\mathbb{R}}
\newcommand{\ZZ}{\mathbb{Z}}
\newcommand{\ZZnz}{\mathbb{Z}_{\neq{}0}}
\newcommand{\addgrp}{\mathbb{G}_\mathrm{a}}

\newcommand{\fo}{\mathfrak{o}}
\newcommand{\charfun}{\mathcal{X}}

\DeclareMathOperator{\Cl}{Cl}
\DeclareMathOperator{\Cone}{Cone}
\DeclareMathOperator{\Eff}{Eff}
\DeclareMathOperator{\EP}{EP}

\DeclareMathOperator{\Frac}{Frac}
\DeclareMathOperator{\Gal}{Gal}

\DeclareMathOperator{\Pic}{Pic}
\DeclareMathOperator{\Proj}{Proj}
\DeclareMathOperator{\rk}{rk}
\DeclareMathOperator{\SAmple}{SAmple}
\DeclareMathOperator{\Spec}{Spec}
\DeclareMathOperator{\V}{V}

\newcommand{\abs}[1]{\left\lvert#1\right\rvert}
\newcommand{\norm}[1]{\left\lVert#1\right\rVert}
\newcommand{\relmiddle}[1]{\mathrel{}\middle#1\mathrel{}}
\newcommand{\fppf}{\mathrm{fppf}}
\newcommand{\irr}{\mathrm{irr}}
\newcommand{\imult}{\mathrm{m}}

\newtheorem{theorem}{Theorem}[section]
\newtheorem{lemma}[theorem]{Lemma}
\newtheorem{proposition}[theorem]{Proposition}
\theoremstyle{definition}

\newtheorem{remark}[theorem]{Remark}

\begin{document}

\title[Integral Points on a log Fano Threefold]{Integral Points of Bounded Height\\ on a log Fano Threefold}

\author{Florian Wilsch}

\address{Institute of Science and Technology Austria,
Am Campus 1, 3400 Klosterneuburg, Austria}

\email{florian.wilsch@ist.ac.at}

\date{November 18, 2021}

\begin{abstract}
We determine an asymptotic formula for the number of integral points of bounded
height on a blow-up of $\PP^3$ outside certain planes using universal torsors.
\end{abstract}

\maketitle

\section{Introduction}
Manin's conjecture~\cite{MR974910,MR1032922} is concerned with the number of rational points on Fano varieties $X$ (that is, smooth, projective varieties with ample anticanonical bundle $\omega_X^\vee$) over a number field $K$ with Zariski dense $K$-rational points.
We may associate height functions $H\colon X(K)\to \RR_{>0}$ with the anticanonical bundle. Manin's conjecture gives a prediction for the number of rational points of bounded anticanonical height that lie in the complement $V$ of all \emph{accumulating} subvarieties, whose rational points would dominate the total number. More precisely, it predicts that the number of rational points of bounded height
\[
  \#\{x\in V(K)\mid H(x) \leq B\}
\]
grows asymptotically as $c B(\log B)^{r-1}$, where $r$ is the Picard number of $X$.

Peyre~\cite{MR1340296,MR2019019} gave a conjectural interpretation of the constant $c$ as a product $\alpha\beta\tau$, where $\alpha$ depends on the geometry of the effective cone, $\beta$ is a cohomological constant connected to the Brauer group and $\tau$ is an adelic volume that can be interpreted as a product of local densities.
Such asymptotics are in particular known for generalized flag varieties~\cite{MR974910}, toric varieties~\cite{MR1620682}, equivariant compactifications of vector groups~\cite{MR1906155}, and some smooth del Pezzo surfaces~\cite{MR1909606,MR2099200,MR2838351}.

Fano threefolds were classified by Iskovskih, Mori and Mukai~\cite{MR463151,MR641971}. For these, Manin proved a lower bound for the number of rational points after a finite extension of the base field~\cite{MR1199203}. Those Fano threefolds that are toric or additive and for which Manin's conjecture is thus known have been classified by Batyrev~\cite{MR631434} and Huang--Montero~\cite{MR4104377}, respectively. Besides such results for general classes of varieties, Manin's conjecture for Fano threefolds remains open.

\smallskip

On proper varieties, integral points on an integral model and rational points coincide as a consequence of the valuative criterion for properness. A set-up concerning integral points on a non-proper variety analogous to Manin's conjecture is the following: Consider a smooth log Fano variety over a number field $K$, by which we shall mean a smooth, projective variety $X$ together with a reduced, effective divisor $D$ with strict normal crossings over an algebraic closure such that the log-anticanonical bundle $\omega_X(D)^\vee$ is ample. Let $H$ be a log-anticanonical height function, let $\mU$ be a flat integral model of $X-D$ and consider the complement $V\subset X$ of all subvarieties whose points would dominate the number of integral points on $\mU$. How does the number of integral points of bounded height
\[
  \#\{x\in \mU(\fo_K) \cap V(K) \mid H(x) \leq B\}
\]
behave asymptotically?

Results in this direction include complete intersections of large dimension compared to their degree~\cite{MR0150129},
algebraic groups and homogeneous spaces~\cite{MR1230289,MR1381987,MR1230290,MR1309971,MR2286635,MR2488484,MR3438314},
and partial equivariant compactifications~\cite{arXiv:1006.3345,MR2999313,MR3117310}, that is, equivariant compactifications $X$ together with an invariant divisor $D$. The first case is an application of the circle method; for the latter cases, the group structure is exploited by means of harmonic analysis or similar methods.

In~\cite{MR2740045}, Chambert-Loir and Tschinkel describe a framework allowing a geometric interpretetation of such asymptotic formulas.
These results suggest that the asymptotic formula for a split variety $X$ (i.e., such that $\Pic(X)\to\Pic(X_{\QQbar})$ is an isomorphism) over the field $K=\QQ$  of rational numbers, with a geometrically integral divisor $D$ admitting a real point, should have the form
\[
  \alpha \tau_{\infty} \tau_{\mathrm{fin}} B (\log B)^{\rk\Pic X -1}(1+o(1))\text{,}
\]
where $\alpha$ depends on the geometry of the effective cone, $\tau_{\infty}$ is a Tamagawa volume of the boundary $D(\RR)$, and $\tau_{\mathrm{fin}}$ is product of local volumes of integral points $\mU(\ZZ_p)$. 

\bigskip

Our main result is such an asymptotic formula for a log Fano threefold that does not belong to any of the above classes. To this end, we parametrize the integral points using universal torsors. Universal torsors have been defined and studied by Colliot-Thélène and Sansuc~\cite{MR899402}; their usage to count rational points goes back to Salberger~\cite{MR1679841}, who used them to reprove Manin's conjecture for toric varieties. Since then, the technique has been used to count rational points on many other varieties. This is the first application of the torsor method to integral points.

We will count integral points on a smooth log Fano threefold $(X,D)$, where $X$ is in particular Fano, has Picard number 2 and is of type 30 in the classification of Fano threefolds~\cite{MR641971}.
Let $\pi\colon X\to\PP^3$ be the blow-up of $\PP^3=\Proj \QQ[a,b,c,d]$ along the smooth conic $C=\V(a^2+bc,d)$.
We will provide asymptotic formulas for the number of integral points on $X-D_i$, where $D_1$ is the preimage $\pi^{-1}(\V(b))$ of a plane intersecting $C$ twice in one rational point and $D_2$ is the preimage $\pi^{-1}(\V(a))$ of a plane intersecting $C$ in two rational points. Up to $\QQ$-automorphism, these are precisely the planes intersecting $C$ in rational points.
(Indeed, if $H=V(f)$ is another such plane, the term of $f$ involving $d$ can be removed by a linear change of variables while leaving $C$ invariant. We get a three-dimensional quadratic form $q(a,b,c)=a^2+bc$ with a $2$-dimensional subspace $H'=\{f=0\}$, with $(H',q)$ isometric to $(\{b=0\},q)$ in the case of one rational point and isometric to $(\{a=0\},q)$ in the case of two rational points. Witt's theorem extends this isometry to a linear transformation involving $a,b,c$ and leaving $q$ invariant.)

To construct integral models $\mU_i$ of $U_i=X-D_i$, we consider the blow-up $\mX$ of $\PP_\ZZ^3$ along $V(a^2+bc,d)$ and define $\mU_1=\mX-\overline{D_1}$, $\mU_2=\mX-\overline{D_2}$.
Manin's conjecture for rational points on this variety is known by~\cite{MR1906155}, since it is a compactification of $\addgrp^3$, so it provides a natural starting point for the investigation of integral points on threefolds by new methods. Note that even though the complete variety $X$ is an equivariant compactification, the open subvarieties $U_i$ whose integral points we are counting are not partial equivariant compactifications, so our result is not a special case of~\cite{MR2999313}. Cf.\ Lemma~\ref{lem:noaction} and Remark~\ref{rmk:compactification} for details.

We describe the sets of integral points explicitly by a universal torsor in Section~\ref{sec:torsor}.
In Section~\ref{sec:metrics}, following definitions in~\cite{MR2740045}, we construct a log-anticanonical height function $H\colon X(\QQ)\to \RR_{>0}$ in \eqref{eq:heightdesc}, measures $\tau_{(X,D_i),p}$ on $X(\QQ_p)$ in \eqref{eq:log-tamagawa} together with convergence factors that turn out to be $\lambda_p = (1-1/p)$, and a measure $\tau_{D_i,\infty}$ on $D_i(\RR)$ (Lemma~\ref{lem:arch-volumes} and before) together with a renormalization factor $c_\RR=2$. We continue with the description \eqref{eq:alpha} of constants $\alpha_i$ and an interpretation of exponent of $\log B$ in the asymptotic. 

In Sections~\ref{sec:case-V-b} and~\ref{sec:case-V-a}, we prove an asymptotic formula for the number of integral points of bounded height on $\mU_1$ and $\mU_2$. In both cases, the exceptional divisor and the strict transform of $V(d)$ are accumulating (Remarks~\ref{rmk:accumulating-1} and~\ref{rmk:accumulating-2}). Defining $V_1 = \pi^{-1} (V(bd))\subset U_1$ and $V_2 = \pi^{-1} (V(ad))\subset U_2$ to be their complements in $U_1$ and $U_2$, respectively, we count the number 
\[
  N_i(B)=\#\{x\in\mU_i(\ZZ)\cap V_i(\QQ)\mid H(x)\leq B\}.
\]
of integral points of height at most $B$ not contained in these accumulating subvarieties, for real numbers $B>1$ and $i\in\{1,2\}$.
A comparison of these formulas with the computations in the preceding section results in the following:

\begin{theorem}
  For $i\in\{1,2\}$, the number of integral points of bounded height satisfies the asymptotic formula
  \[
    N_i(B)=\alpha_i \tau_{i,\mathrm{fin}} \tau_{i,\infty} B \log B (1+o(1))\text{,}
  \]
  where
  \begin{align*}
    \tau_{i,\mathrm{fin}}&=\prod_p \lambda_p \tau_{(X,D_i),p}(\mU_i(\ZZ_p)) \qquad \text{and}\\
    \tau_{i,\infty}&=c_\RR\tau_{D_i,\infty}(D_i(\RR))\text{.}
  \end{align*}
  More explicitly, we have
  \begin{align*}
    N_1(B)&=\frac{20}{3\zeta(2)}B\log B + O(B) \qquad \text{and}\\
    N_2(B)&=\frac{20}{3}\prod_p\left(1-\frac{2}{p^2}+\frac{1}{p^3}\right) B\log B + O(B(\log\log B)^2)\text{.}
  \end{align*}
\end{theorem}

\section{A Universal Torsor}\label{sec:torsor}
The Cox ring of $X$ over $\QQbar$ is by definition
\[
  R(X_\QQbar)=\bigoplus_{d\in\Pic (X_\QQbar)} H^0(X,\cL_d),
\]
where $(\cL_d)_d$ is a suitable system of representatives of every class in the geometric Picard group; its ring structure is induced by the sum and tensor product of sections. By~\cite[Theorem 4.5, Case 30]{MR3348473}, it is
\[
  R(X_{\QQbar})=\QQbar[a,b,c,x,y,z]/(a^2+bc-yz).
\]
The Picard group of $X$ is $\Pic(X_\QQbar)\cong\Pic(X)\cong\ZZ^2$ with a basis given by the classes of the pull-back $H$ of a plane in $\PP^3$ and the exceptional divisor $E$. Note that the anticanonical bundle is $\omega_X^\vee \cong 4\cO_X(4H-E)$.
The first three generators are sections cutting out the preimages of the coordinate hyperplanes $V(a)$, $V(b)$, and $V(c)$ in $\PP^3$,
respectively. The generator $x$ cuts out the strict transform of $V(d)$, the generator $y$ the strict transform of $V(a^2+bc)$, and $z$ cuts out the exceptional divisor.
This leads to the $\Pic(X)$-grading 
\begin{center}
\begin{tabular}{cccccc}\toprule
  $a$ & $b$ & $c$ & $ x$ & $ y$ & $z$ \\ \midrule
  $1$ & $1$ & $1$ & $ 1$ & $ 2$ & $0$ \\
  $0$ & $0$ & $0$ & $-1$ & $-1$ & $1$ \\\bottomrule 
\end{tabular}
\end{center}
of the Cox ring, expressed in the above basis.
(In loc.~cit., a different basis of the Picard group is used. Moreover, the fourth and fifth generators are mistakenly swapped in the equation, making it inhomogeneous, an impossibility.)

\begin{lemma}\label{lem:torsor-closure}
  The variety
  \[
    T_\QQbar=\Spec R(X_\QQbar) - V(I_\irr),
  \]
   where $I_\irr=(a,b,c,z)(x,y)$, is a universal torsor over $X_\QQbar$.
\end{lemma}
\begin{proof}
  In addition to the ring itself, we argue using the \emph{bunch of cones} $\Phi$ associated with $X$, following~\cite[3.2]{MR3307753}. Consider the bunch
  \[
    \Phi = \{\Cone(\{\deg(t)\mid t\in M\}) \mid M \subset \{a,b,c,x,y,z\} \ \text{s.t. \ref{enum:solubility} and \ref{enum:canonical}  hold} \}
  \]
  of cones in $\Pic(X)_\RR$. Here, $M$ runs over subsets of the generators satisfying
  \begin{enumerate}[label=(\roman*)]
    \item\label{enum:solubility} $\prod_{t\in M} t \not\in \sqrt{(t\mid t\not\in M)}$,
    that is, the equation $a^2+bc-yz$ has a solution with $t=0$ for $t\not\in M$ and $t\neq 0$ for $t\in M$, and
    \item\label{enum:canonical} $\omega_X^\vee\in\Cone(\{\deg(t)\mid t\in M\})$.
  \end{enumerate}
  Concretely, this bunch of cones is
  \begin{equation}\label{eq:phi-concrete}
    \Phi=\left\{
    \Cone\left([\substack{1\\0}],[\substack{1\\-1}]\right),
    \Cone\left([\substack{0\\1}],[\substack{1\\-1}]\right),
    \Cone\left([\substack{1\\0}],[\substack{2\\-1}]\right),
    \Cone\left([\substack{0\\1}],[\substack{2\\-1}]\right)
    \right\} \text{,}
  \end{equation}
  given by, for example, the generators $\{b,x\}$, $\{z,x\}$, $\{a,y\}$, and $\{a,y,z\}$, respectively. Note that these are all possible cones containing the anticanonical bundle and that condition~\ref{enum:solubility} can be seen to hold by considering the solutions $(0,1,0,1,0,0)$, $(0,0,0,1,0,1)$, $(0,1,0,0,1,0)$, and $(1,0,0,0,1,1)$, respectively.  In particular, $\Phi$ is a \emph{true bunch} in the sense of \cite[Definition~3.2.1.1\,(ii),\,(iii)]{MR3307753}).

  By~\cite[Theorem~3.2.1.9~(ii)]{MR3307753}, $X$ is defined by a bunched ring with a maximal bunch $\Phi'$. As a consequence of the description of the ample cone in~\cite[Proposition~3.3.2.6]{MR3307753}, every cone in $\Phi'$ needs to contain the (ample) anticanonical class. It follows that $\Phi'\subset \Phi$, whence $\Phi'=\Phi$ by maximality.

  Following~\cite[Construction~3.2.1.3]{MR3307753}, set
  \begin{equation}\label{eq:torsor-constr}
    T_\QQbar =\widehat X_\QQbar = \bigcup_{\substack{M\subset (a,b,c,x,y,z) \text{ s.t.}\\\Cone(\deg(t)\mid t\in M)\in \Phi}} \left(\Spec R(X_\QQbar) - V\left(\prod_{t\in M} t \right)\right).
  \end{equation}
  By~\cite[Theorem~3.2.1.4]{MR3307753}, it is a characteristic space of $X_\QQbar$, and by~\cite[Proposition~6.1.3.9.\,(ii)]{MR3307753}, it is a universal torsor over $X_\QQbar$.
  The construction~\eqref{eq:torsor-constr} can be rewritten as $T_\QQbar = \Spec R(_\QQbar) - V(I_\irr)$, where $I_\irr$ is the \emph{irrelevant ideal}, generated by all elements of the form $\prod_{t\in M} t$ such that $M$ is a subset of the generators satisfying $\Cone(\deg(t)\mid t\in M)\in \Phi$. This yields
  \[
    I_\irr=(ax,bx,cx,zx,ay,by,cy,zy)=(a,b,c,z)(x,y),
  \]
  since the minimal subsets $M$ suffice.
\end{proof}

Next, we construct an integral model of this torsor. Consider the ring
\begin{equation}\label{eq:R_Z}
  R_\ZZ=\ZZ[a,b,c,x,y,z]/(a^2+bc-yz)
\end{equation}
and the ideal $I_{\irr,\ZZ}=(a,b,c,z)(x,y)\subset R_\ZZ$.
\begin{lemma}\label{lem:torsor-model}
  The scheme $\mT=\Spec R_\ZZ - V(I_{\irr,\ZZ})$ is a $\GG_{\imult,\ZZ}^2$-torsor over $\mX$.
\end{lemma}
\begin{proof}
  We first note that removing $yz$ from the set generators of $I_\irr$ does not change the radical of the ideal. The degrees of the two factors of any of the remaining generators
  \begin{equation}\label{eq:f_i}
    f_1= ax,\ f_2=bx,\ f_3 = cx,\ f_4=zx,\ f_5=ay,\ f_6=by,\ f_7=cy 
  \end{equation}
  form a basis of the Picard group; in particular, every element of the Picard group is the degree of an element of the form $a^{k_1}x^{k_2}\in R_\ZZ[f_1^{-1}]$, for some $k_1,k_2\in\ZZ$, and can analogously be written as the degree of an element in $R_\ZZ[f_i^{-1}]$ for $i\in\{2,\dots,7\}$.  
  Thus, it follows from~\cite[Theorem 3.3]{MR3552013} that $\mT =\Spec R_\ZZ - V(I_{\irr,\ZZ})$ is a $\GG_{\imult,\ZZ}^2$-torsor over the $\ZZ$-scheme $\mX'$ obtained by gluing the spectra
  \begin{equation}\label{eq:mU_i}
    \mV_i  =\Spec R_\ZZ[f_i^{-1}]^{(0)}, \quad i\in\{1,\dots,7\},
  \end{equation}
  of the degree-$0$-parts of the localizations of $R_\ZZ$ in the generators $f_i$ of the irrelevant ideal.
  This integral model $\mX'$ of $X_\QQbar$ coincides with the blow-up $\mX$.
  Indeed, we can embed both the Cox ring $R_\ZZ$ and the Rees algebra
  \[
    A=\bigoplus_{n\geq 0} I^n = \ZZ[a,b,c,d][(a^2+bc)\xi,d\xi]
  \]
  for $I=(a^2+bc,d)$ into the field $\QQ(a,b,c,d,\xi)=\Frac(A)$, where the first embedding maps $z\mapsto \xi^{-1}$, $x\mapsto d\xi$, and $y\mapsto (a^2+bc)\xi$.
  The blow-up is then given by gluing the spectra of the seven rings $A_{s,t}\subset\Frac(A)$ arising the following way:
  First take the degree-0-part (with respect to the usual grading of $\ZZ[a,b,c,d]$, not considering the natural grading of the Rees algebra) of the localizations of $A$ in $s\in\{a,b,c,d\}$, then further localize in one of the generators $t\in\{\frac{a^2+bc}{s^2}\xi$, $\frac{d}{s}\xi\}$ ($t=\xi$ suffices for $s=d$) of the Rees algebra and take the degree-0-part with respect to the grading induced by the natural grading of the Rees algebra.
  The rings $R_\ZZ[f^{-1}]^{(0)}$ for $f$ in $ax,bx,cx,zx,ay,by,cy$ coincide with the rings $A_{s,t}$ for $(s,t)$ in
  \[
    (a,d\xi/a),(b,d\xi/b),(c,d\xi/c),(d,\xi),(a,(a^2+bc)\xi),(b,(a^2+bc)\xi),(c,(a^2+bc)\xi),
  \]
  so the two schemes defined by the blow-up and~\cite[Construction 3.1]{MR3552013} coincide.
\end{proof}

Denote by $p\colon \mT\to \mX$ a morphism rendering $\mT$ such a torsor. We note that the composition of morphisms $T\to X \to \PP^3$ (on the generic fibers) maps
\begin{equation*}
  (a,b,c,x,y,z)\to\pc{a:b:c:xz}.
\end{equation*}
It follows from the observations at the beginning of this section that $V(x)\subset T$ is the preimage of the strict transform of $V(d)\subset \PP^3$, that $V(y)\subset T$ is the preimage of the strict transform of $V(a^2+bc)$, and that $V(z)\subset T$ is the preimage of the exceptional divisor $E\subset X$.

\begin{lemma}\label{lem:int-points-on-torsor}
  The morphism $p$ induces a $4$-to-$1$-correspondence between integral points on $\mX$ and
  \begin{equation}\label{eq:int-points-on-torsor}
    \mT(\ZZ)=\left\{ (a,b,c,x,y,z)\in\ZZ^6 \relmiddle|
    \substack{a^2+bc-yz=0\\ \gcd(a,b,c,z)=\gcd(x,y)=1} \right\} \text{,}
  \end{equation}
  between integral points on $\mU_1$ and
  \begin{equation}\label{eq:int-points-on-torsor-b}
    \mT_1(\ZZ)=\left\{ (a,b,c,x,y,z)\in\ZZ^6 \relmiddle|
    \substack{a^2+bc-yz=0\\ b=\pm 1, \gcd(x,y)=1} \right\}\text{,}
  \end{equation}
  and between integral points on $\mU_2$ and
  \begin{equation}\label{eq:int-points-on-torsor-a}
    \mT_2(\ZZ)=\left\{ (a,b,c,x,y,z)\in\ZZ^6 \relmiddle|
    \substack{a^2+bc-yz=0\\ a=\pm 1, \gcd(x,y)=1} \right\}\text{.}
  \end{equation}
\end{lemma}
\begin{proof}
  The fiber $f^{-1}(P)$ of any point $P\in\mX(\ZZ)$ is a $\GG_{\imult,\ZZ}^2$-torsor. Since such torsors are parameterized by $H^2_\fppf(\Spec\ZZ,\GG_{\imult}^2)=\Cl(\ZZ)^2=1$, all fibers are isomorphic to $\GG_{\imult,\ZZ}^2$, and we get a $4$-to-$1$-correspondence between integral points on the torsor $\mT$ and those on $\mX$.

  Since $\mT$ is quasi-affine, its integral points have a description as lattice points satisfying the equation of the Cox ring and coprimality conditions given by the irrelevant ideal. Points on the preimages of $\mU_1$ and $\mU_2$ under the morphism $p\colon \mT\to\mX$ are defined by the additional condition $(b)=1$ and $(a)=1$, respectively.
\end{proof}

We conclude this section with some observations on the geometry of $X$.

\begin{lemma}\label{lem:noaction}
  There is no action of $\addgrp^3$ on $X$ with an open orbit under which $D_1$ or $D_2$ are invariant, neither is $X$ toric.
\end{lemma}
\begin{proof}
  For the first part, assume for contradiction that there is such an action that leaves one of the $D_i$ invariant. The action of $\addgrp^3$ on $X$ induces a continuous, whence trivial, action of $\addgrp^3(\QQ)=\QQ^3$ on $\Pic X$. Hence, the exceptional divisor is invariant under the action, which therefore restricts to $X-E \cong \PP^3 - C$. As $D_i$ is invariant, too, the action further restricts to the complement $W_i$ of $E\cup D_i$, which is isomorphic to $\PP^3 -(C\cup H_i) \cong \bA^3-C_0$ (for one of the planes $H_1=V(b),\ H_2=V(a)$)---the complement of the conic $C_0=C\cap \bA^3$ in $\bA^3$. Since the action has an open orbit $U\cong \addgrp^3$ by assumption, we get an open immersion $\bA^3\cong U \hookrightarrow W_i\subset \bA^3-C$, an impossibility by Ax--Grothendieck.

  For the second part, we note that the Cox ring of $X$ is not a polynomial ring, while all toric varieties have polynomial rings as their Cox rings, cf.\ \cite{MR1786494}.
\end{proof}

\begin{remark}\label{rmk:compactification}
  The total variety $X$ is a compactification of $\addgrp^3$, as classified by Huang and Montero~\cite{MR4104377} (induced by the action of $\addgrp^3$ on $\PP^3$, where the group acts trivially on the plane $V(d)$ and by addition on the complement). Manin's conjecture for rational points~\cite{MR1906155} and asymptotics for integral points on some open subvarieties~\cite{MR2999313} are known due to Chambert-Loir and Tschinkel: The admissible divisors $D$ are the exceptional divisor, the strict transform of $V(d)$, and their sum.
  Even though $X$ is an equivariant compactification of $\addgrp^3$, the pairs $(X,D_i)$ are neither partial equivariant compactifications of $\addgrp^3$ nor toric by the previous lemma. Our result is thus not a special case of~\cite{arXiv:1006.3345} or~\cite{MR2999313}.
\end{remark}

Lastly, we can describe the geometric Picard group with the information we gathered in the proof of Lemma~\ref{lem:torsor-closure}: The pseudo-effective cone is generated by the degrees of the generators of the Cox ring, so $\overline{\Eff}(X)=\Cone(E,H-E)$, where $H$ is the pull-back of the class of a plane in $\PP^3$. The semi-ample cone is the intersection of all cones in $\Phi$ as in \eqref{eq:phi-concrete} and thus $\SAmple(X)=\Cone(H,2H-E)$.
In particular, the log-anticanonical bundles
\[
  \omega(D_1)^\vee\cong \omega(D_2)^\vee\cong\cO_X(3H-E)
\]
are in its interior, hence ample.

\section{Metrics, Heights, Tamagawa Measures, and Predictions}\label{sec:metrics}

\subsection{Adelic metrics}
To construct a log-anticanonical height function, we endow certain line bundles with adelic metrics. For fixed $d\in\Pic(X)$, the elements of degree $d$ in the Cox rings are the global sections of a line bundle $\cL_d$ with isomorphism class $d$ (such that $\cL_d\otimes \cL_e = \cL_{d+e}$ by the construction of the Cox ring).
Fixing models of these bundles will be helpful: For $d\in \ZZ^2 \cong \Pic X$, consider the $\cO_\mX$-module $\mL_d$ obtained by gluing the degree-$d$-parts of $R_\ZZ[f_i^{-1}]$ (as modules over the degree-$0$-parts), where $R_\ZZ$ and $f_i$ as in \eqref{eq:R_Z} and \eqref{eq:f_i}. It is locally trivialized by multiplication with an element of the form $a^{k_1}x^{k_2}$ of degree $d$ on $\mV_1$ as in \eqref{eq:mU_i}, whose existence was checked in the proof of Lemma~\ref{lem:torsor-model}, and similarly on the remaining affine opens $\mV_2,\dots, \mV_7$ covering $\mX$.

The sheaf $\mL_{[3,-1]}$ is generated by the global sections
\[
  s_1 = a^2x,\ s_2 = b^2x,\ s_3 = c^2x,\ s_4 = z^2x^3,\ s_5 = ay,\ s_6 = by,\ \text{and}\ s_7 = cy,
\]
each $s_i$ being a generator on $\mV_i$, while, similarly, $\mL_{[1,0]}$ is globally generated by
$\{a,b,c,xz\}$ and $\cL_{[4,-1]}$ by the $28$ pairwise products of these sections. 
The first set of section induces a morphism $\mX\to \PP_\ZZ^{6}$ with $\mL_{[3,-1]}\cong \cO_{\PP^6_\ZZ(1)}$ and the other two sets analogous morphisms. This way, we get an adelic metric (in the sense of e.g.\ \cite[Définition~2.3]{MR2019019})
\begin{equation}\label{eq:log-metrization}
  (s,\pc{a:b:c:x:y:z}) \mapsto \frac{\abs{s(a,b,c,x,y,z)}_v}{\max\{\abs{a^2x}_v,\abs{b^2x}_v,\abs{c^2x}_v,\abs{z^2x^3}_v,\abs{ay}_v,\abs{by}_v,\abs{cy}_v\}},
\end{equation}
on $\cL_{[3,-1]}$,
where $\pc{a:b:c:x:y:z}$ is the image of $(a,b,c,x,y,z)\in T(\QQ_v)$ (i.e., a point in \emph{Cox coordinates}) in $X(\QQ_v)$ and $s \in \Frac(R(X))$ has degree $[3,-1]$, regarded as a meromorphic section of $\cL_{[3,-1]}$, and an analogously defined metric on $\cL_{[1,0]}$.
These metrics, as well as the product metric on $\cL_{[4,-1]}$, are induced by the respective models $\mL_d$ at all finite places, as the same is true for the canonical metrics on $\cO_{\PP^n}(1)$ for all $n\ge 1$.

\subsection{A log-anticanonical height function}
The metric on $\cL_{[3,-1]}\cong \omega_X^\vee$ induces the log-anticanonical height function
\[
  H\colon V(K)\to \RR_{\geq 0}, \quad x\mapsto \prod_v \norm{s(x)}^{-1}_v\text{,}
\]
where $s$ is a section that does not vanish in $x$. Since $\prod_v\abs{\alpha}_v=1$ for all $\alpha\in \QQ$, this does not depend on the choice of $s$. Since $X$ is proper, every rational point in $X(\QQ)$ lifts to a unique integral point in $\mX(\ZZ)$, which in turn corresponds to four integral points $(a,b,c,x,y,z)\in \mT(\ZZ)$ by Lemma~\ref{lem:int-points-on-torsor}.
By the coprimality condition and the equation, no prime can divide all of the monomials in the denominator of~\eqref{eq:log-metrization}. Thus we get
\begin{equation}\label{eq:heightdesc}
  H\pc{a:b:c:x:y:z}=\max\left\{\abs{a^2x},\abs{b^2x},\abs{c^2x},\abs{z^2x^3},\abs{ay},\abs{by},\abs{cy}\right\}
\end{equation}
for the image $\pc{a:b:c:x:y:z}\in X(\QQ)$ of $(a,b,c,x,y,z)\in \mT(\ZZ)$ (with the usual real absolute value).

\subsection{Tamagawa measures}

To explicitly calculate Tamagawa volumes, we need metrics on the bundles $\omega_X$, $\cO_x(D_1)$, and $\cO_X(D_2)$, not just on bundles isomorphic to them. It will turn out to be helpful to choose isomorphisms spreading out to the integral models of the line bundles constructed above. To this end, we start by noting that as $\mX$ is smooth, $\omega_\mX$ is invertible and a model of $\omega_X$. 

As $\Cl(\QQ) = 1$ (and $\mX$ is smooth), the map $\Pic \mX \to \Pic X$ is injective. Hence, $\omega_\mX^\vee\cong \mL_{[4,-1]}$ and $\cO_\mX(\overline{D_1})\cong \cO_\mX(\overline{D_2}) \cong \mL_{[1,0]}$.
To choose an explicit isomorphism, we note that, up to a unit, the canonical section $1_{D_1}$ (resp.\ $1_{D_2}$) is the unique primitive section of $\cO_\mX(\overline{D_1})$ (resp. $\cO_\mX(\overline{D_2})$) cutting out $\overline{D_1}$ (resp.\ $\overline{D_2}$). This also holds for the elements $b$ (resp. $a$) of the degree-$[1,0]$-part of the Cox ring (regarded as the global sections of the bundle $\mL_{[1,0]}$), so there exists an isomorphisms with $1_{D_1}\mapsto b$ (resp.\ $1_{D_2}\mapsto a$), and we shall use this isomorphism.
For the (anti-)canonical bundle, we consider the chart
\begin{equation}\label{eq:chart}
f\colon V\to \bA^3,\ \pc{a:b:c:x:y:z}\mapsto \left(\frac{a}{xz},\frac{b}{xz},\frac{c}{xz}\right)
\end{equation}
and its inverse
\begin{equation}\label{eq:chart-inverse}
 g\colon \bA^3\to V,\ (a_0,b_0,c_0)\mapsto \pc{a_0:b_0:c_0:1:a_0^2+b_0c_0:1} \text{,}
\end{equation}
where $V=X-V(xz)=\pi^{-1}(V(d))\cong\bA^3$, both spreading out to the integral model (more precisely, to an isomorphism $\mV_4\to \bA_\ZZ^3$ and its inverse, with $\mV_4$ as in \eqref{eq:mU_i}). Denote by $a_0$, $b_0$, and $c_0$ the coordinate functions on $\bA^3$ and their compositions with $f$.
The sections $\diff a_0 \wedge \diff b_0 \wedge \diff c_0$ and $\frac{\diff}{\diff a_0}\wedge \frac{\diff}{\diff b_0} \wedge \frac{\diff}{\diff c_0}$ of the canonical and anticanonical bundle have neither zeroes or poles on $\bA^3\cong V$, and their tensor product is $1$.
Up to a unit, they are the only primitive sections with this property. Since the analogous property holds for $x^{-4}z^{-3}$ and $x^4z^3$, we can fix isomorphisms identifying $\diff a_0 \wedge \diff b_0 \wedge \diff c_0$ with $x^{-4}z^{-3}$ and $\frac{\diff}{\diff a_0}\wedge \frac{\diff}{\diff b_0} \wedge \frac{\diff}{\diff c_0}$ with $x^4z^3$.
In particular, these isomorphisms induce adelic metrics on $\omega_X$, $\omega_X^\vee$, $\cO_X(D_1)$, and $\cO_X(D_2)$.

Recall that the adelic metric on $\omega_X$ determines a \emph{Tamagawa measure} $\tau_{X,v}$ on the $\QQ_v$-points $X(\QQ_v)$ for all places $v$.
In the local coordinates $a_0,b_0,c_0$, it is given by $\diff \tau_{X,v} = \norm{\diff a_0 \wedge \diff b_0 \wedge \diff c_0}^{-1}_v \diff \mu_v$, where $\mu_p$ is the Haar measure satisfying $\mu_p(\ZZ_p)=1$ for finite places $p$ and the Lebesgue measure $\diff\mu_\infty=\diff x$ for the archimedean place.
In the context of integral points, a modified measure that has been defined by Chambert-Loir and Tschinkel in~\cite[2.1.10]{MR2740045} tends to appear in asymptotic formulas:
The metric on $\cO_X(D_i)$ induces another measure $\tau_{(X,D_i),v}$ on $X(\QQ_v)$ on setting
\begin{equation}\label{eq:log-tamagawa}
  \diff \tau_{(X,D_i),v}=\norm{1_{D_i}}_{\cO_X(D_i)}^{-1}\diff\tau_{X,v},
\end{equation}
where $1_{D_i}$ is the canonical section of $\cO_X(D_i)$.

\subsubsection*{Finite places}

\begin{lemma}
  For any prime $p$ and $i\in\{1,2\}$, we have
  \begin{equation*}
    \tau_{(X,D_i),p}(\mU_{i}(\ZZ_p)) = \frac{\#\mU_i(\FF_p)}{p^3} = 
    \begin{cases}
      1+\frac{1}{p}               &\quad \text{if $i=1$,} \\
      1+\frac{1}{p}-\frac{1}{p^2} &\quad \text{if $i=2$.} 
    \end{cases}
  \end{equation*}
\end{lemma}
\begin{proof}
  As the metric on $\omega_X$ is induced by $\omega_\mX$ at all finite places and $\mX$ is smooth, it follows from \cite[Corollary~2.15]{MR1679841} that $\tau_X(\mU_i(\ZZ_p))=\frac{\#\mU_i(\FF_p)}{p^3}$. From an analogous argument, we can deduce that $\norm{1_{D_i}}_{\cO(D_i),p}=1$ for all $p$, whence $\tau_{(X,D_i),p}(\mU_i(\ZZ_p))=\tau_{X,p}(\mU_i(\ZZ_p))$.

  To compute the number of points modulo $p$, we start by observing that
  \[
    \#\mX(\FF_p) = \#\PP_\ZZ^3(\FF_p) + p^2 + p =  p^3 + 2p^2 + 2p + 1;
  \]
  indeed, the reduction of the conic $\overline{C}$ modulo $p$ is split, whence isomorphic to $\PP_{\FF_p}^1$ and has $p+1$ points, each of which is replaced by a projective line with $p+1$ points on the blow-up. The planes
  $V(b)$ and $V(a)$ whose preimages constitute the boundaries $\overline{D_1}$ and $\overline{D_2}$ have $p^2+p+1$ points each. The point (resp.\ two points) in the intersection of $V(a)$ (resp.\ $V(b)$) with $C$ is again replaced by $p+1$ points (each), whence
  \[
    \#\overline{D_1}(\FF_p) = p^2 + 2 p + 1
    \qquad \text{and} \qquad
    \#\overline{D_2}(\FF_p) = p^2 + 3 p + 1,
  \]
  and the assertion on the number of $\FF_p$-points on the open subschemes $\mU_i = \mX-\overline{D_i}$ follows.
\end{proof}

\subsubsection*{Archimedean place}

The metric on $\omega_X(D_i)$ induces a metric on the canonical bundle $\omega_{D_i}$ of $D_i$ via the adjunction isomorphism.
This metric induces a \emph{residue measure} $\tau_{D_i,v}$ on $D_i(\QQ_v)$ for any place $v$.
(The process could be repeated to define measures on intersections of components of reducible divisors.)
See~\cite[2.1.12]{MR2740045} for details.

\begin{lemma}\label{lem:arch-volumes}
  We have $\tau_{D_1,\infty}(D_1(\RR))=\tau_{D_2,\infty}(D_2(\RR))=20$.
\end{lemma}
\begin{proof}
  Following the constructions and descriptions in loc.~cit., the adjunction isomorphism induces a metric on $\omega_{D_1}$ via
  \begin{equation}\label{eq:log-antic-norm}
      \norm{\diff a_0 \wedge \diff c_0}_{\omega_{D_1}} = \norm{\diff a_0 \wedge \diff b_0 \wedge \diff c_0}_{\omega_X}\norm{b_0}_{\cO_X(-D_1)}^{-1},
  \end{equation}
  using local coordinates coming from the chart \eqref{eq:chart}.
  Since $\diff a_0 \wedge \diff b_0 \wedge \diff c_0$ corresponds to $x^{-4}z^{-3}\in R(X)$, the first factor of \eqref{eq:log-antic-norm} is
  \begin{align*}
    &\frac{\max\{\abs{a^2x},\abs{b^2x},\abs{c^2x},\abs{z^2x^3},\abs{ay},\abs{by}, \abs{cy}\} \max\{\abs{a},\abs{b},\abs{c},\abs{xz}\}}{\abs{x^4z^3}}\\
    &\qquad =\max\{\abs{a_0^2},\abs{c_0^2},1,\abs{a_0^3}, \abs{a_0^2c_0}\}\max\{\abs{a_0},\abs{c_0},1\}\text{,}
  \end{align*}
  when evaluated in $g(a_0,0,c_0)\in V\cap D_1$. On the affine variety $V$, regarding $b_0=b/xz$ as an element of $\Gamma(V,\cO_V(-D_1))\subset \cO_V(V)$ and using the canonical trivialization of $\cO(-D_1)$ outside $D_1$, we get
  \[
    \norm{b_0}^{-1}_{\cO_X(-D_1)}
    = \lim_{b_0\to 0} \left(\frac{\abs{b_0}}{\norm{1_{D_1}}_{\cO_X(D_1)}} \right)^{-1},
  \]
  using the continuity of the metric. As $1_{D_1}$ corresponds to $b\in R(X)$ under our chosen isomorphism, the expression inside the limit is
  \[
    \left( \frac{ \abs{b} \max\{\abs{a},\abs{b},\abs{c},\abs{xz}\}}{\abs{bxz}} \right)^{-1}
  \]
  in Cox coordinates. Evaluating at $g(a_0,b_0,c_0)$ along \eqref{eq:chart-inverse} and taking the limit results in
  $\max\{\abs{a_0},\abs{c_0},1\}^{-1}$. We thus have explicit descriptions
  \begin{equation}\label{eq:tamagawa-explicit}
    \diff f_*\tau_{D_1,\infty} = \norm{\diff a_0 \wedge \diff c_0}_{\omega_{D_1}}^{-1}
    \diff a_0 \diff c_0 =
    \frac{1}{\max\{\abs{a_0^2},\abs{c_0^2},1,\abs{a_0^3}, \abs{a_0^2c_0}\}} \diff a_0 \diff c_0
  \end{equation}
  and, by an analogous argument,
  \[
    \diff f_*\tau_{D_2,\infty} = 
    \norm{\diff b_0 \wedge \diff c_0}_{\omega_{D_2}}^{-1} \diff b \diff c = \frac{1}{\max\{\abs{b_0^2},\abs{c_0^2},1,\abs{b_0^2c_0}, \abs{b_0c_0^2}\}} \diff b_0 \diff c_0
  \]
  of the Tamagawa measures $\tau_{D_1,\infty}$ and $\tau_{D_2,\infty}$ with respect to the Lebesgue measure.

 To compute the volume of the first divisor, we integrate \eqref{eq:tamagawa-explicit} and perform a change of variables $a=a_0$, $b=b_0$ for notational convenience, getting
  \begin{align*}
    \tau_{D_1,\infty}(D_1(\RR))
    =
    &\int_{\abs{a},\abs{a^2c}\leq 1} \frac{1}{\max\{\abs{c^2},1\}} \diff a \diff c
    + \int_{\substack{\abs{a}\geq\abs{c}\\\abs{a}> 1}} \frac{1}{\abs{a^3}} \diff a \diff c \\ 
     &\qquad + \int_{\substack{\abs{c}>\abs{a}\\\abs{a^2c}>1}} \frac{1}{\max\{\abs{c^2},\abs{a^2c}\}} \diff a \diff c \text{.}
  \end{align*}
  The first term of this expression is $\frac{20}{3}$ by \eqref{eq:integral-case-V-b} below, the second is $\int_{\abs{a}}\frac{2}{\abs{a^2}}=4$ and the third is
  \begin{equation}\label{eq:integral-inf-place-2-step}
    \int_{\substack{\abs{c}>\abs{a} , \abs{a^2c}>1 \\ \abs{a^2} > \abs{c}}}
    \frac{1}{\abs{a^2c}} \diff a \diff c
    + \int_{\substack{\abs{c}>a , \abs{a^2c}>1 \\ \abs{a^2}\leq \abs{c}}} \frac{1}{\abs{c^2}}\diff a \diff c
    \text{.}
  \end{equation}
  In \eqref{eq:integral-inf-place-2-step}, the first term is
  \begin{equation*}
    \int_{\substack{\abs{c}\geq 1 \\ \abs{c}^{1/2} < \abs{a} < \abs{c}}} \frac{1}{\abs{a^2c}} \diff a \diff c=
    \int_{\abs{c}\geq 1} \frac{2}{\abs{c}} (\abs{c}^{1/2} - \abs{c}^{-1}) \diff c=4
  \end{equation*}
  and the second is
  \begin{equation*}
    \int_{a\in\RR} \frac{2}{\max\{\abs{a^2},\abs{a^{-2}}\}} \diff a
    = \int_{\abs{a}\leq 1} 2\abs{a}^2 \diff a
    + \int_{\abs{a}>1} \frac{2}{\abs{a^2}} \diff a
    = \frac{16}{3}\text{.}
  \end{equation*}
  Thus, \eqref{eq:integral-inf-place-2-step} is $\frac{28}{3}$ and $\tau_{D_1,\infty}(D_1(\RR))=\frac{20}{3} + 4 + \frac{28}{3}=20$.

  For the other divisor, we get $\tau_{D_2,\infty}(D_2(\RR))=20$ by similar arguments.
\end{proof}

\subsubsection*{Convergence Factors}

Following \cite[2.4]{MR2740045}, these measures are renormalized with factors associated with the virtual Galois module
\[
  \EP(U_i) =
  [(\QQbar[U_i]^\times/\QQbar^\times) \otimes \QQ] -[\Pic U_{\overline{\QQ}} \otimes \QQ]
\]
at the finite places.
Since both $U_1$ and $U_2$ have only constant nowhere vanishing global sections over any algebraically closed field
and the Galois group acts trivially on the geometric Picard groups, this leads to the trivial module $\EP(U_i) = - [\QQ]$ for both $i=1,2$. Its Artin $L$-function is
\[
  L(s, -[\QQ]) = \prod_p (1-1/p^s) = \zeta^{-1}(s),
\]
and evaluating its factors at $s=1$ leads to the convergence factors
\[
  \lambda_p=1-\frac{1}{p}.
\]
In general, the resulting product of measures is multiplied with the principal value of the $L$-function; however, this is simply $\lim_{s\to 1}(s-1)^{-1}\zeta(s)^{-1}=1$ in our case.

The residue measures at the infinite places are renormalized by factors depending on the fields of definition of the divisor components. For a geometrically integral divisor $D$, this is simply a factor of $c_\RR=2$ at any real place; see Sections 3.1.1 and 4.1 of~\cite{MR2740045} for details.

\subsection{The constant \texorpdfstring{$\alpha$}{alpha}}
Previous results such as~\cite{MR2999313} suggest that in the case of a number field with only one infinite place and a geometrically irreducible divisor $D$, a factor $\alpha$ completely analogous to Peyre's for rational points should appear in asymptotic formulas, replacing the anticanonical bundle by the log-anticanonical bundle in \cite[Définition~4.8]{MR2019019}:
Consider the pseudo-effective cone $\overline\Eff(X)\subset\Pic(X)_\RR$ and its characteristic function $\charfun_{\overline\Eff(X)}(\cL)=\int_{\overline\Eff(X)^\vee} e^{-\langle \cL,t\rangle}\diff t$ (with respect to the Haar measure on $\Pic(X)^\vee_\RR$ normalized by $\Pic(X)^\vee$).
Then define
\begin{equation}\label{eq:alpha}
  \alpha_i=\frac{1}{(\rk\Pic(X)-1)!} \charfun_{\overline\Eff(X)}(\omega_X(D_i)^\vee).
\end{equation}

In our case, $\overline\Eff(X)=\Cone(E,H-E)$ is spanned by a basis of $\Pic(X)$ and $\omega_X(D_i)^\vee\cong\cO_X(3(H-E)+2E)$; hence, $\alpha_i=1/6$ for both $i=1,2$ can be easily computed using \cite[Proposition~5.3~(ii)]{MR1620682}.

\subsection{The exponent of \texorpdfstring{$\log B$}{log B}}
In previous work such as \cite{arXiv:1006.3345,MR2999313}, the exponent of $\log B$ in asymptotic formulas is $b-1$, where
\[
  b = -\rk \EP(U_i)^G + d,
\]
for $G=\Gal(\QQ)$, and where $d$ depends on incidence properties of the boundary divisor $D$, as encoded in the \emph{Clemens complex}.
In our case of geometrically integral $D_i$, the Galois invariant part of the virtual module has rank
$\rk\EP(U_i)^G = -\rk\Pic U_i = 1-\rk\Pic X$,
while the \emph{Clemens complex} just consists of a vertex, resulting in $d=1$. In particular, $b=\rk\Pic X$ in this case.

\section{Integral Points on \texorpdfstring{$X-D_1$}{X - D\_1}}\label{sec:case-V-b}

We study the number
\[
  N_1(B)=\#\{x\in\mU_1(\ZZ)\cap V_1(\QQ)\mid H(x)\leq B\}
\]
of integral points of bounded height on $\mU_1 = \mX - \overline{V(b)}$ that, as rational points, are in the complement $V_1$ of $V(bxz)=\pi^{-1}(V(bd))$.

Using the $4$-to-$1$-correspondence \eqref{eq:int-points-on-torsor-b} with integral points on the universal torsor $\mT_1$ and noticing the symmetry in the two values $\pm 1$ of $b$ in \eqref{eq:int-points-on-torsor-b}, this description of integral points on the universal torsor yields the formula
\begin{equation*}
  N_1(B)=\frac{1}{2} \, \# \left\{(a,c,x,y,z)\in \ZZ^5 \relmiddle| \substack{
    a^2+c-yz=0\text{, } \gcd(x,y)=1,\\
    H(a,1,c,x,y,z)\leq B\text{, }
    x,z\neq 0
  }  \right\}\text{,}
\end{equation*}
where
\[
  H(a,b,c,x,y,z)=\max{\{ \abs{a^2x},\abs{b^2x},\abs{c^2x},\abs{z^2x^3},\abs{ay},\abs{by},\abs{cy}\}}
\]
by~\eqref{eq:heightdesc}.
Solving the equation, we can simplify this to
\begin{equation*}
  \frac{1}{2} \, \# \left\{(a,x,y,z)\in \ZZ^4 \relmiddle|
  \substack{
     \gcd(x,y)=1\text{, } \widetilde H_1(a,x,y,z)\leq B, \\
     x,z\neq 0
  } \right\}\text{,}
\end{equation*}
where
\begin{align*}
    \widetilde H_1(a,x,y,z)&=H(a,1,yz-a^2,x,y,z) \\
    &=\max\{\abs{a^2x},\abs{x},\abs{(yz-a^2)^2x},\abs{z^2x^3},\abs{ay},\abs{y},\abs{(yz-a^2)y}\}.
\end{align*}

\begin{lemma}
  We have
  \[
    N_1(B)=\frac{1}{2}\sum_{\alpha>0} \frac{\mu(\alpha)}{\alpha} \sum_{x',z\in\ZZnz}
    \int_{\substack{\abs{a^2 \alpha x'},\abs{a(a^2+c)z^{-1}},\\ \abs{c^2 \alpha x'},\abs{c(a^2+c)z^{-1}}, \\ \abs{\alpha^3x'^3z^2} \leq B }} \frac{1}{\abs{z}} \diff a \diff c + O(B)\text{.}
  \]
\end{lemma}
\begin{proof}
  A Möbius inversion yields
  \[
    N_1(B)=\sum_{\alpha>0}\mu(\alpha)\sum_{\substack{a\in\ZZ\\ x',z \in \ZZnz}}\#\{y' \in \mathbb{Z} \mid \widetilde H_1(a,\alpha x', \alpha y',z)\leq B \}\text{.}
  \]
  We get
  \[
    \#\{y^\prime \in \mathbb{Z} \mid \widetilde H_1(a,\alpha x', \alpha y',z)\leq B \} = V_1(\alpha,a,x',z;B) + O(1) \text{,}
  \]
  where
  \[
    V_1(\alpha,a,x',z;B) = \int_{\widetilde H_1(a,\alpha x', \alpha y',z)\leq B} \diff y'.
  \]
  Note that by its definition, $V_1(\alpha,a,x',z;B)=0$ whenever the variables $\alpha$, $a$, $x'$, and $z$ violate a height condition involving only them. 
  In the following steps, we get similar estimates with an error term of the form $O(C\sup_\xi f(\xi))$ when replacing the sum over $\xi$ of a non-negative function $f$ that is piecewise differentiable and whose derivative changes sign at most $C$ times by an integral; by~\cite[Lemma 3.6]{MR3269462}, such a constant $C$ exists and is independent of the remaining variables, resulting in the error $O(\sup_\xi f(\xi))$. 
  We can bound the sum over the error term by
  \[
    \ll \sum_{\substack{\alpha >0,\ a\in\ZZ,\ x',z \in\ZZnz\\\abs{\alpha a^2 x'}, \abs{\alpha x'}, \abs{\alpha^3 z^2x'^3} \le B}} 1 \ 
    \ll \sum_{\substack{\alpha >0,\ x'\in\ZZnz\\ \abs{\alpha x'}\le B}} \frac{B}{\lvert \alpha x'\rvert^2}
    \ll B,
  \]
  after noting that the remaining height condition implies $B^{1/2}/\abs{\alpha z'}^{1/2}\gg 1$ when estimating the sum over $a$ so that we can ignore the additional $O(1)$ error term arising from the possibility of $a$ being $0$.
  Hence, 
  \[
    N_1(B)= \sum_{\alpha>0}\mu(\alpha)\sum_{a\in \ZZ,\ x',z\in\ZZnz} \int_{\widetilde H_1(a,\alpha x', \alpha y',z)\leq B} \diff y' + O(B).
  \]

  Turning to the variable $a$ next we estimate the sum $\sum_{a\in\ZZ} V_1(\alpha,a,x',z;B)$ by the integral
  \[
    V_2(\alpha,x',z;B)=\int_{a \in \RR} V_1(\alpha,a,x',z;B)\diff a,
  \]
  introducing an error bounded by
  \[
    \begin{aligned}
      &\ll \sum_{\substack{\alpha>0,\ x',z\in\ZZnz\\\abs{\alpha^3x'^3z^2}\le B}} \sup_{a\in\ZZ} V_1(\alpha,a,x',z;B)
      \ll \sum_{\substack{\alpha>0,\ x',z\in\ZZnz\\\abs{\alpha^3x'^3z^2}\le B}} \frac{B^{1/2}}{\alpha^{3/2} \abs{x'}^{1/2}\abs{z}} \\
      &\ll \sum_{\substack{\alpha>0,\ z\in\ZZnz}} \frac{B^{2/3}}{\alpha^2 \abs{z}^{4/3}} \ll B^{2/3} \text{,}
    \end{aligned}
  \]
  where we use the condition $\abs{(\alpha y'z-a^2)^2\alpha x'}\leq B$ to estimate the integral $V_1$.
  A change of variable $c=\alpha y' z-a^2$ now results in the description
  \[
    V_2(\alpha,x',z;B)
    = \int_{\substack{\abs{a^2 \alpha x'},\abs{a(a^2+c)z^{-1}},\\ \abs{c^2 \alpha x'},\abs{c(a^2+c)z^{-1}}, \\ \abs{\alpha^3x'^3z^2} \leq B }} \frac{1}{\abs{\alpha z}} \diff a \diff c
  \]
  of the main term.
\end{proof}

\begin{lemma}
  We have
  \begin{equation} \label{eq:count-2-2}
    N_1(B) = \frac{1}{2} \sum_{\alpha>0}\frac{\mu(\alpha)}{\alpha^2}
    \int_{\substack{\abs{a^2 x},\abs{a^3z^{-1}}, \abs{c^2x},\\ \abs{a^2c z^{-1}}, \abs{x^3z^2}\leq B\\\abs{z}\geq 1,
  \abs{x}\geq \alpha }} \frac{1}{\abs{z}} \diff a \diff c \diff x \diff z + O(B) \text{.}
  \end{equation}
\end{lemma}
\begin{proof}
  We first want to replace the two instances of $a^2+c$ by $a^2$ in the inequalities defining the region for the volume function $V_2$ of the previous lemma, to get a new volume function $V_2'(\alpha,x',z;B)$.
  The error we introduce when replacing $\abs{a(a^2+c)z^{-1}}$ by $\abs{a^3z^{-1}}$ is bounded by the integral over the region
  \[
    B-\abs{\frac{ac}{z}}\leq \abs{\frac{a^3}{z}} \leq B+\abs{\frac{ac}{z}}\text{,} \qquad \text{i.e.,} \qquad
    \abs{a^2-\frac{B\abs{z}}{\abs{a}}}\leq\abs{c}\text{.}
  \]
  With a change of variable $a'=a^2-B\abs{z}\abs{a}^{-1}$, where
  \[
    \abs{\frac{\diff a'}{\diff a}}=2\abs{a}+\frac{B\abs{z}}{\abs{a}^2}\geq
     \sqrt{\abs{a'}}\text{,}
  \]
  we can bound the total error by
  \[
    \begin{aligned}
      &\ll \sum_{\alpha>0}\frac{1}{\alpha} \sum_{x',z\in\ZZnz}
      \int_{\substack{\abs{a'}\leq\abs{c},\abs{\alpha c^2x'},\\\abs{\alpha^3 x'^3z^2}\leq B}} \frac{1}{\sqrt{\abs{a'}}\abs{z}} \diff a' \diff c\\
      & \ll
      \sum_{\alpha>0}\frac{1}{\alpha} \sum_{x',z\in\ZZnz} \int_{\substack{\abs{\alpha^3 x'^3 z^2},\\\abs{\alpha c^2 x'}\leq B}} \frac{\sqrt{\abs{c}}}{\abs{z}} \diff c
      \ll \sum_{\alpha>0} \frac{1}{\alpha} \sum_{\substack{x',z\in\ZZnz \\\abs{\alpha^3 x'^3 z^2} \leq B}} \frac{B^{3/4}}{\alpha^{3/4}\abs{x}^{3/4}\abs{z}}\\
      & \ll \sum_{\substack{\alpha>0 \\ z\in\ZZnz}} \frac{B^{5/6}}{\alpha^{7/4}\abs{z}^{7/6}}
      \ll B^{5/6}\text{.}
    \end{aligned}
  \]
  When modifying the other inequality, the error we introduce is bounded by an integral over a similar region, and, after an analogous change of variable, we get the same bound.

  Next, we estimate the summation over $z$. Using the height conditions $\abs{a}\leq B^{1/3}\abs{z}^{1/3}$ and $\abs{c}\leq B^{1/2}\abs{\alpha x}^{-1/2}$, we can bound the volume
  \[
    V_2'(\alpha,x',z;B) \ll \frac{B^{5/6}}{\abs{\alpha x}^{1/2}\abs{z}^{2/3}}\text{.}
  \]
  Replacing the sum over $z$ by an integral, we introduce an error
  \[
    \ll \sum_{\alpha>0} \frac{1}{\alpha^{3/2}} \sum_{1\leq \abs{x'} \leq B^{1/3}}\frac{B^{5/6}}{\abs{x}^{1/2}} \ll B \text{.}
  \]
  For $V_3(\alpha,x';B)=\int_{\abs{z}\geq 1} V_2'(a,x',z;B)\diff z$, we get an upper bound
  \[
    V_3(\alpha,x';B)\ll \int_{\abs{\alpha^3x'^3z^2}\leq B} \frac{B^{5/6}}{\alpha^{1/2}\abs{x'}^{1/2}\abs{z}^{2/3}} \diff z
    \ll \frac{B}{\alpha \abs{x'}}\text{.}
  \]
  Finally, replacing the sum over $x'$ by an integral
  $\int_{\abs{x'}\geq 1} V_3(\alpha,x';B)$ introduces an error term
  \[
    \ll \sum_{\alpha>0} \frac{B}{\alpha^2} \ll B,
  \]
  and a change of variables $x=\alpha x'$ completes the proof.
\end{proof}

\begin{proposition}
  The number of integral points of bounded height on $\mU_1$ satisfies the asymptotic formula
  \[
  N_1(B)=\frac{20}{3\zeta(2)} B\log B+O(B)\text{.}
  \]
\end{proposition}
\begin{proof}
  By a change of variables
  \begin{equation*}
  a\mapsto a z^{-1/3}B^{-1/3},\quad c\mapsto c z^{-1/3}B^{-1/3},\quad x\mapsto x z^{2/3}B^{-1/3},
  \end{equation*}
  we get
  \begin{align}
      N_1(B) & =\frac{1}{2} \sum_{\alpha>0} \frac{\mu(\alpha)}{\alpha^2}  \int_{\substack{\abs{a^2x}, \abs{c^2x}, \abs{a},\\ \abs{a^2c}, \abs{x}\leq 1,\\ 1\leq\abs{z}\leq\abs{x}^{3/2}B^{1/2}\alpha^{-3/2}}} \frac{B}{\abs{z}}  \diff a \diff c \diff x \diff z + O(B) \nonumber \\
      & = \frac{1}{2} \sum_{\alpha>0} \frac{\mu(\alpha)}{\alpha^2} \int_{\substack{\abs{a^2x}, \abs{c^2x}, \abs{a},\\ \abs{a^2c}, \abs{x}\leq 1}}
      B\left(\log\left(B \abs{x}^3 \alpha^{-3}  \right) \right) \diff a \diff c \diff x + O(B),\label{eq:sum-estimate-N_1} 
  \end{align}
  as the error introduced by omitting the condition $1\le \abs{x}^{3/2}B^{1/2}\alpha^{-3/2}$ is
  \[
    \begin{aligned}
      &\ll \sum_{\alpha>0} \frac{B}{\alpha^2} \int_{\substack{\abs{x}^{3/2}\alpha^{-3/2}B^{1/2},\\ \abs{a},\abs{c^2x} \leq 1}}
      \abs{\log\left(B \abs{x}^3 \alpha^{-3} \right)} \diff a
       \diff c \diff x \\
      & \ll \sum_{\alpha>0} \frac{B}{\alpha^2}
      \int_{\abs{x}\leq \alpha B^{-1/3}}
      zx\abs{\log\left(B^{1/3}\abs{x}\alpha^{-1}\right)} \frac{1}{\abs{x}^{1/2}} \diff x 
      \ll \sum_{\alpha>0} \frac{B}{\alpha^2} \frac{\alpha^{1/2}}{B^{1/6}}  \ll B^{5/6}.
    \end{aligned}
  \]
  Removing the factor $\abs{x}^3\alpha^{-3}$ in the logarithm in~\eqref{eq:sum-estimate-N_1} results in
  \[
    N_1(B) = \frac{1}{2} \sum_{\alpha>0} \frac{\mu(\alpha)B \log B}{\alpha^2} \int_{\substack{\abs{a^2x}, \abs{c^2x}, \abs{a},\\ \abs{a^2c}, \abs{x}\leq 1}} \diff a \diff c \diff x + O(B),
  \]
  as the error introduced this way is 
  \[
    \begin{aligned}
      & \ll \sum_{\alpha > 0} \frac{B}{\alpha^2} \int_{\abs{a},\abs{c^2x},\abs{x}\leq 1} 
      \abs{\log\left( \frac{\abs{x}}{\alpha} \right)} \diff a \diff c \diff x \\
      & \ll \sum_{\alpha > 0} \frac{B}{\alpha^2} \int_{\abs{x} \leq 1} 
      \frac{1}{\abs{x}^{1/2}} \abs{\log\left( \frac{\abs{x}}{\alpha} \right)} \diff x
       \ll \sum_{\alpha > 0} \frac{B}{\alpha^2} (2+ \log(\alpha))
       \ll B \text{.}
    \end{aligned}
  \]
  Now, integrating over $x$ results in
  \[
    N_1(B) = \frac{B \log B}{\zeta(2)} \int_{\substack{\abs{a},\abs{a^2c}\leq 1}} \frac{1}{\max\{1,\abs{a^2},\abs{c^2}\}} \diff a \diff c + O(B).
  \]
  Finally, we note that the integral evaluates to
  \begin{equation}\label{eq:integral-case-V-b}
    \int_{\substack{\abs{a},\abs{a^2c}\leq 1}} \frac{1}{\max\{1,\abs{c^2}\}} \diff a \diff c
    =\int_{\abs{c}\leq 1} 2 \diff c + \int_{\abs{c}>1} \frac{2}{\abs{c}^{5/2}}\diff c 
    = \frac{20}{3} \text{,}
  \end{equation}
  and arrive at the asymptotic expression.
\end{proof}

\begin{remark}\label{rmk:accumulating-1}
  The strict transform of $V(d)$ (corresponding to points with $x=0$ on the universal torsor) and exceptional divisor (corresponding to points with $z=0$) are accumulating. For all $B>0$, they contain the images of 
  \begin{align*}
    &\{(a,1,c,0,1,a^2+bc) \in\ZZ^6\mid \abs{a},\abs{c}\le B,\ ac\ne 0\}\subset \mT_1(\ZZ) \quad \text{and} \\ 
    & \left\{
      (a,1,-a^2,x,y,0)\in\ZZ^6 \ \middle|\ 
      \substack{
        \abs{a}\le B^{1/8},\ \abs{x},\abs{y}\le B^{1/2},\\
        axy\ne 0,\ \gcd(x,y)=1
      }
    \right\}\subset \mT_1(\ZZ),
  \end{align*}
  respectively.
  All points in these sets have height at most $B$, and each set contains $\gg B^{9/8}$ points.
\end{remark}

\section{Integral Points on \texorpdfstring{$X-D_2$}{X - D\_2}}\label{sec:case-V-a}

We count the number
\[
  N_2(B)=\#\{x\in \mU_2(\ZZ)\cap V_2(\QQ)\mid H(x)\leq B\}
\]
of integral points of bounded height on $\mU_2=\mX - \overline{V(a)}$, that, as rational points, are in the complement $V_2$ of $V(axz)=\pi^{-1}(V(ad))$.

\begin{lemma}
  We have
  \begin{equation}\label{eqn:countingproblem1}
    N_2(B)=\frac{1}{2} \, \# \left\{(b,c,x,y,z)\in \ZZ^5 \relmiddle|
    \substack{
      1+bc-yz=0\text{, } \gcd(x,y)=1, \\ H(1,b,c,x,y,z)\leq B,\\b,c,x,z\neq 0
    }
    \right\} + O(B)\text{.}
  \end{equation}
\end{lemma}
\begin{proof}
  With the $4$-to-$1$-correspondence to integral points on the torsor, and noticing the symmetry in the two possible values $a=\pm 1$ of $a$ in~\eqref{eq:int-points-on-torsor-a}, we get the expression up to the missing conditions $b,c\ne 0$, which will help in the following arguments. We note that if $b=0$, the torsor equation reads
  $1-yz=0$, implying $y,z\in \{\pm 1\}$. Using the height conditions $\abs{a^2x},\,\abs{c^2x}\le B$ (and $a^2=1$), we get
  \begin{align*}
    &\#\{(a,0,c,x,y,z) \in \mT_2(\ZZ) \mid H(a,b,c,x,y,z)\le B, x\ne 0\} \\
    & \qquad\qquad\ll \#\{ (\pm 1,0,c,x,\pm 1,\pm 1)\in \ZZ^6 \mid x\ne 0, \abs{c^2x}\le B \} \\
    & \qquad\qquad\ll \sum_{x\in \ZZnz, \abs{x}\le B} \frac{B^{1/2}}{\abs{x}^{1/2}} \ll B.
  \end{align*}
  Hence, we can add the condition $b\ne 0$ found in \eqref{eqn:countingproblem1}, introducing an error $O(B)$, and by an analogous argument, we can add $c\ne 0$.
\end{proof}

\begin{lemma}
  We have
  \begin{equation*}
    N_2(B)=\sum_{b,x,z\in\ZZnz}\theta_1(b,x,z)V_1(b,x,z;B) + O(B)\text{,}
  \end{equation*}
  where
  \[
      V_1(b,x,z;B)=\frac{1}{2}\int_{\substack{\widetilde H_2(b,c,x,z)\leq B\\\abs{b},\abs{c},\abs{x},\abs{z}\ge 1}}  \frac{1}{\abs{z}} \diff c
  \]
  with
  \begin{align*}
    \widetilde H_2 (b,c,x,z)
    & = H(1,b,c,x,(1+bc)z^{-1},z)\\
    & =\max\left\{\abs{x},\abs{b^2x},\abs{c^2x},\abs{z^2x^3},
    \abs{\frac{(1+bc)}{z}},\abs{\frac{b(1+bc)}{z}},\abs{\frac{c(1+bc)}{z}}\right\}\text{,}
  \end{align*}
  and $\theta_1(b,x,z)=\prod_p\theta_1^{(p)}(b,x,z)$ with
  \[
    \theta_1^{(p)}(b,x,z)= \begin{cases}
      0 & \text{if}\quad p\mid b,\ p\mid z \text{,}\\
      1-\frac{1}{p} & \text{if}\quad p\nmid b,\ p\mid x,  \\
      1 & \text{otherwise.}
    \end{cases}
  \]
\end{lemma}
\begin{proof}
  Using a Möbius inversion to remove the condition $\gcd(x,y)=1$ in~\eqref{eqn:countingproblem1},
  and setting $y'=\frac{y}{\alpha}$, we get
  \[
    N_2(B) =
    \frac{1}{2} \sum_{b,x,z\in \ZZnz} \sum_{\alpha\mid x}
    \mu(\alpha) \widetilde{N_2}(\alpha,b,x,z;B) \text{,}
  \]
  where
  \[
    \widetilde{N_2}(\alpha,b,x,z;B)=
    \#\left\{(c,y')\in \ZZ^2 \relmiddle|
    \substack{
      c\neq 0\text{, } 1+bc-y'\alpha z=0,\\
      H(1,b,c,x,\alpha y',z)\leq B
    } \right\}\text{.}
  \]
  To estimate $\widetilde N_2$, we first note that $\widetilde{N_2}(\alpha,b,x,z;B)=0$ whenever $\alpha z$ and $b$ are not coprime.
  If they are coprime, we estimate
  \[
    \begin{aligned}
      \widetilde{N_2}(\alpha,b,x,z;B) 
      & = \# \left\{ c \in \ZZnz \relmiddle|
      \substack{
        bc \equiv -1 \pmod{\alpha z},\\
        \widetilde H_2(b,c,x,z)\leq B
      }
      \right\} \\
      & = \int_{\substack{\widetilde H_2(b,c,x,z)\leq B\\\abs{c}\ge 1}} \frac{1}{\abs{\alpha z}} \diff c+ O(1) \text{,}
    \end{aligned}
  \]
  analogously to the first case.
  This inequality together with the height conditions $\abs{b^2x}\leq B$ and $\abs{z^2x^3}\leq B$ allows us to bound the summation over the error terms by
  \begin{displaymath}
    \ll \sum_{\substack{b,x,z\in\ZZnz\\ \abs{b^2x}, \abs{z^2x^3}\leq B}}\sum_{\alpha | x } \abs{\mu(\alpha)}
    \ll \sum_{x\in\ZZnz} \frac{2^{\omega(x)} B}{\abs{x}^2} \ll
    B \text{,}
  \end{displaymath}
  using that $2^{\omega(x)}\ll_\varepsilon x^\varepsilon$ for all $\varepsilon>0$. We arrive at
  \begin{equation*}
    N_2(B)=\sum_{b,x,z}\sum_{\substack{\alpha | x\\ \gcd(b, \alpha   z)=1}}\frac{\mu(\alpha)}{\alpha} V_1(b,x,z;B) + O(B)\text{,}
  \end{equation*}
  where
  \[
    V_1(b,x,z;B)=\frac{1}{2}\int_{\substack{\widetilde H_2(b,c,x,z)\leq B\\\abs{b},\abs{c},\abs{x},\abs{z}\ge 1}}  \frac{1}{\abs{z}} \diff c \text{.}
  \]
  Using the multiplicativity of $\mu$ and $\gcd$, we can factor the sum over $\alpha$
  \[
    \begin{aligned}
      \sum_{\substack{\alpha | x\\\gcd(b, \alpha   z)=1}}
      \frac{\mu(\alpha)}{\alpha}
       = \prod_p\begin{cases}
        0 & \text{if}\quad p\mid b,\ p\mid z\text{,}\\
        1-\frac{1}{p} & \text{if}\quad  p\nmid b,\ p\mid x, \\
        1 & \text{otherwise}
      \end{cases}
    \end{aligned}
  \]
  to get a description of the arithmetic term $\theta_1$.
\end{proof}

\begin{lemma}
  We have
  \begin{equation*}\label{eq:count-1-2}
    N_2(B)=\sum_{b,z} \theta_2(x,z) V_2(x,z;B) + O(B\left(\log\log B)^2\right)\text{,}
  \end{equation*}
  where
  \[
      V_2(x,z;B)=\frac{1}{2}\int_{\substack{\widetilde H_2(b,c,x,z)\leq   B\\\abs{b},\abs{c},\abs{x},\abs{z}\ge 1}}  \frac{1}{\abs{z}} \diff b \diff c
  \]
  and $\theta_2(x,z)= \prod_p \theta_2^{(p)}(x,z)$ with
  \[
  \theta_2^{(p)}(x,z)=\begin{cases}
    (1-\frac{1}{p})^2 &\text{if}\quad p \mid x,z\text{,}\\
    1-\frac{1}{p}+\frac{1}{p^2} 
                      &\text{if}\quad p \mid x,\ p\nmid z\text{,}\\
    1-\frac{1}{p}     &\text{if}\quad p\nmid x,\ p\mid z\text{,}\\
    1                 &\text{if}\quad p\nmid xz\text{.}
  \end{cases}
  \]
\end{lemma}
\begin{proof}
  Using the height conditions $\abs{c^2x},\abs{b(1+bc)z^{-1}}\leq B$ to estimate the integral, we can bound the volume function by the geometric average
  \[
    \begin{aligned}
    V_1(b,x,z;B) & \ll \frac{1}{\abs{z}}   \left(\frac{B^{1/2}}{\abs{x}^{1/2}}\right)^{2/3}
    \left(\frac{B \abs{z}}{\abs{b}^2} \right)^{1/3} \\
    & \ll \frac{B}{\abs{bxz}} \left(\frac{B}{\abs{b^2x}}\right)^{-1/6}
    \left(\frac{B}{\abs{z^2x^3}}\right)^{-1/6} \text{.}
    \end{aligned}
  \]
  Since the integral is zero whenever $\abs{b^2x}\geq B$ or $\abs{z^2x^3}\geq B$, the assertion follows by~\cite[Proposition 3.9]{MR2520770} with $r=0$, $s=2$.
  (In the notation of loc.\ cit.\ we consider the ordering $\eta_0=b,\eta_1=x,\eta_2=z$ of the variables, take $a_1=a_2=1/6$, and $k_{i,j}$ to be the exponents in these two height conditions. Note that $\theta_1$ satisfies~\cite[Definition 7.8]{MR2520770}, and hence the requirements of the proposition.)
\end{proof}

\begin{lemma}\label{lem:N_2-by-volume}
  We have
  \begin{equation*}
    N_2(B)= \frac{1}{2} \prod_p \left(1-\frac{2}{p^2}+\frac{1}{p^3}\right) \int_{\substack{\widetilde H_2(b,c,x,z)\leq   B\\\abs{b},\abs{c},\abs{x},\abs{z}\ge 1}}  \frac{1}{\abs{z}} \diff  b \diff c \diff x \diff z + O(B(\log\log B)^2) \text{.}
  \end{equation*}
\end{lemma}
\begin{proof}
  Using the same estimate for the integral over $c$ as in the previous lemma and estimating the integral over $b$ using the height condition $\abs{b^2x}\leq B$, we get the bound
  \[
    \begin{aligned}
      V_2(x,z;B)
      & \ll \int_{1\leq\abs{b}\leq B^{1/2}\abs{x}^{-1/2}} \frac{B^{2/3}}{\abs{b}^{2/3} \abs{x}^{1/3} \abs{z}^{2/3}} \\
      & \ll \frac{B}{\abs{xz}} \left(\frac{B}{\abs{x}^{3} \abs{z}^{2}}\right)^{-1/6}
    \end{aligned}
  \]
  for the volume function $V_2$. Since $V_2(b,z;B)=0$ whenever $\abs{z^2x^3}>B$, we get an asymptotic formula by~\cite[Proposition 4.3]{MR2520770} (with $r=s=1$). We are only left to see that the constant is indeed
  \begin{align*}
    &\prod_p\left(\frac{1}{p^2}\left(1-\frac{1}{p}\right)^2+\frac{1}{p}\left(1-\frac{1}{p}\right)\left(2-\frac{2}{p}+\frac{1}{p^2}\right) + \left(1-\frac{1}{p}\right)^2\right) \\
    & \qquad\qquad = \prod_p \left(1-\frac{2}{p^2}+\frac{1}{p^3}\right) \text{.}\qedhere
  \end{align*}
\end{proof}

\begin{proposition}
  We have
  \[
    N_2(B)=c B \log(B)+ O(B(\log\log B)^2),
  \]
  where
  \[
    c=\frac{20}{3}\prod_p \left(1-\frac{2}{p^2}+\frac{1}{p^3}\right)  \text{.}
  \]
\end{proposition}
\begin{proof}
  We have to estimate the integral in Lemma~\ref{lem:N_2-by-volume}. We first want to replace $(1+bc)$ by $bc$ in the height conditions. In the case of the condition $b(1+bc)/z$, this leaves us with an error term that can be bounded by the integral over the region defined by
  $B-\abs{\frac{b}{z}}\leq \abs{\frac{b^2c}{z}}\leq B+\abs{\frac{b}{z}}$, i.e.,
  $\abs{\frac{Bz}{b^2}}-\frac{1}{\abs{b}} \leq \abs{c} \leq \abs{\frac{Bz}{b^2}}+\frac{1}{\abs{b}}$,
  and the remaining height conditions, hence is at most
  \begin{equation}\label{eq:bound-1+bc}
    \ll\int_{\substack{\abs{b^2x},\abs{x^2z^3}\le B\\\abs{b},\abs{z}\geq 1}} \frac{1}{\abs{bz}}\diff b \diff x \diff z. 
  \end{equation}
  Using the condition $\abs{x}^{3/2}\abs{z}^{3/2}\abs{b}\le B$, which is implied as the geometric mean of the others, this error is bounded by
  \[
    \ll
    \int_{\abs{b},\abs{z}\ge 1} \frac{B^{2/3}}{\abs{b}^{5/3}\abs{z}^{2}} \diff b \diff z \ll B^{2/3}\text{.}
  \]
  The condition $c(1+bc)/z$ can be dealt with analogously. Next, we remove the condition $\abs{b}\geq 1$, where we get an error term
  \[
    \begin{aligned}
      & \ll \int_{\substack{\abs{c}\leq\frac{B^{1/2}}{\abs{x^{1/2}}},\abs{x}\leq\frac{B^{1/3}}{\abs{z}^{2/3}}\\ \abs{z}\geq 1 }} \frac{1}{\abs{z}}\diff c \diff x \diff z
        \ll \int_{\substack{\abs{x}\leq\frac{B^{1/3}}{\abs{z}^{2/3}}\\ \abs{z}\geq 1 }} \frac{B^{1/2}}{\abs{x}^{1/2}\abs{z}}\diff x \diff z \\
      & \ll \int_{\abs{z}\geq 1} \frac{B^{2/3}}{\abs{z}^{4/3}}\diff z \ll B^{2/3} \text{,}
    \end{aligned}
  \]
  and subsequently remove $\abs{c}\geq 1$ analogously. Thus, we can estimate the integral in the previous lemma as $V_3(B) + O(B^{2/3})$, where
  \[
    V_3(B)=\int_{\substack{\abs{b^2x},\abs{c^2x},\abs{x^3z^2},\\
    \abs{b^2cz^{-1}},\abs{bc^2z^{-1}} \leq B,\\
    \abs{x},\abs{z}\geq 1 }} \frac{1}{\abs{z}} \diff b \diff c \diff x \diff z \text{.}
  \]
  By a change of variables $b\mapsto B^{-1/3}bz^{-1/3}$, $c \mapsto B^{-1/3}cz^{-1/3}$, $x\mapsto B^{-1/3}xz^{2/3}$, we get
  \begin{equation*}
    V_3(B)
    = B \int_{\substack{\abs{b^2x},\abs{c^2x},\abs{x},\\
      \abs{b^2c},\abs{bc^2} \leq 1,\\
      1\leq \abs{z} \leq B^{1/2}\abs{x}^{3/2} }}
      \frac{1}{\abs{z}} \diff b \diff c \diff x \diff z.
  \end{equation*}
  Integrating over $z$ yields
  \begin{equation}\label{eq:V_a-step-2}
    V_3(B) = 2 B \int_{\substack{\abs{b^2x},\abs{c^2x},\abs{x},\\
      \abs{b^2c},\abs{bc^2} \leq 1}} \log\left(B^{1/2}\abs{x}^{3/2}\right) \diff b \diff c \diff x + O\left(B^{2/3}\right)
  \end{equation}
  since ignoring the bound $1\le B^{1/2}\abs{x}^{3/2}$ introduces an error that is indeed at most
  \[
    \begin{aligned}
      &\ll B \int_{\substack{\abs{x}\leq B^{-1/3} \\ \abs{b^2c},\abs{bc^2}\leq 1 }} \abs{\log(B^{1/2}\abs{x}^{3/2})} \diff b \diff c \diff x 
       \ll B \int_{\abs{b^2c},\abs{bc^2}\leq 1} B^{-1/3} \diff b \diff c \\
      & \ll B^{2/3} \left(\int_{\abs{c}\leq 1} \frac{1}{\sqrt{\abs{c}}} \diff c + \int_{\abs{c}>1} \frac{1}{c^2} \diff c\right) 
      \ll B^{2/3}.
    \end{aligned}
  \]
  Removing $\abs{x}^{3/2}$ from the logarithm in~\eqref{eq:V_a-step-2} results in an error that is bounded by
  \[
    \ll B \int_{\substack{ \abs{b^2c}, \abs{bc^2},\\ \abs{x} \leq 1}} \abs{\log\left(\abs{x}^{3/2}\right)} \diff b \diff c \diff x
    \ll B \int_{\abs{b^2c},\abs{bc^2}\leq 1}\diff b \diff c
    \ll B\text{,}
  \]
  whence
  \[
    V_2(B) = B \log B \int_{\substack{\abs{b^2x},\abs{c^2x},\abs{x},\\
    \abs{b^2c},\abs{bc^2} \leq 1}} \diff b \diff c \diff x + O(B).
  \]
  Integrating over $x$, we get 
  \begin{equation}\label{eq:case-V-a-almost}
    V_2(B)= 2 B\log B \int_{\abs{b^2c},\abs{bc^2} \leq 1} \max\{\abs{b^2},\abs{c^2},1\}^{-1} \diff b \diff c+ O(B)\text{,}
  \end{equation}
  and after computing the integral
  \begin{equation*}
    \begin{aligned}
      \int_{\abs{b^2c},\abs{bc^2} \leq 1} \frac{1}{\max\{\abs{b^2},\abs{c^2},1\}}\diff b \diff c
      & = 2 \int_{\substack{\abs{b}\geq\abs{c}\\ \abs{b^2c}\leq 1}} \frac{1}{\max\{\abs{b^2},1\}} \diff b \diff c \\
      &= 2\int_{b\in\RR} \frac{2\min\{\abs{b},\abs{b^{-2}}\}}{\max\{\abs{b^2},1\}}\diff b \\
      & = \int_{\abs{b}\leq 1} 4\abs{b}\diff b
      + \int_{\abs{b}>1}\frac{4}{\abs{b^4}} \diff b = \frac{20}{3}
    \end{aligned}
  \end{equation*}
  in~\eqref{eq:case-V-a-almost}, we arrive at the desired asymptotic formula.
\end{proof}

\begin{remark}\label{rmk:accumulating-2}
  The strict transform of $V(d)$ (corresponding to points with $x=0$ on the universal torsor) and the exceptional divisor (corrresponding to points with $z=0$) are accumulating. They contain the images of the sets
  \begin{align*}
    &\{(1,b,c,0,1,1+bc)\in\ZZ^6 \mid \abs{b},\abs{c}\le B,\ bc\ne 0\} \subset \mT_2(\ZZ)\quad  \text{and} \\
    &\{(1,1,-1,x,y,0)\in\ZZ^6   \mid \abs{x},\abs{y}\le B,\ \gcd(x,y)=1,\ xy\ne 0\} \subset \mT_2(\ZZ),
  \end{align*}
  respectively. All points in these sets have height at most $B$, and each set contains $\gg B^2$ points.
\end{remark}

\section*{Funding}
This work was supported by the German Academic Exchange Service.

\section*{Acknowledgements}
Parts of this article were prepared at the Institut de Ma\-thé\-ma\-thi\-ques de Jussieu -- Paris Rive Gauche. I wish to thank Antoine Chambert-Loir for his remarks and the institute for its hospitality, as well as the anonymous referee for several useful remarks and suggestions for improvements.


\begin{thebibliography}{0}

  \bibitem{MR3307753}
  I.~Arzhantsev, U.~Derenthal, J.~Hausen, and A.~Laface.
  \newblock {\em Cox rings}, volume 144 of {\em Cambridge Studies in Advanced
    Mathematics}.
  \newblock Cambridge University Press, Cambridge, 2015.
  
  \bibitem{MR631434}
  V.~V. Batyrev.
  \newblock Toric {F}ano threefolds.
  \newblock {\em Izv. Akad. Nauk SSSR Ser. Mat.}, 45(4):704--717, 927, 1981.
  
  \bibitem{MR1032922}
  V.~V. Batyrev and Y.~I. Manin.
  \newblock Sur le nombre des points rationnels de hauteur born\'ee des
    vari\'et\'es alg\'ebriques.
  \newblock {\em Math. Ann.}, 286(1-3):27--43, 1990.
  
  \bibitem{MR1620682}
  V.~V. Batyrev and Y.~Tschinkel.
  \newblock Manin's conjecture for toric varieties.
  \newblock {\em J. Algebraic Geom.}, 7(1):15--53, 1998.
  
  \bibitem{MR0150129}
  B.~J. Birch.
  \newblock Forms in many variables.
  \newblock {\em Proc. Roy. Soc. Ser. A}, 265:245--263, 1961/1962.
  
  \bibitem{MR1309971}
  M.~Borovoi and Z.~Rudnick.
  \newblock Hardy-{L}ittlewood varieties and semisimple groups.
  \newblock {\em Invent. Math.}, 119(1):37--66, 1995.
  
  \bibitem{MR1909606}
  R.~de~la~Bret\`eche.
  \newblock Nombre de points de hauteur born\'ee sur les surfaces de del {P}ezzo
    de degr\'e 5.
  \newblock {\em Duke Math. J.}, 113(3):421--464, 2002.
  
  \bibitem{MR2838351}
  R.~de~la~Bret\`eche and T.~Browning.
  \newblock Manin's conjecture for quartic del {P}ezzo surfaces with a conic
    fibration.
  \newblock {\em Duke Math. J.}, 160(1):1--69, 2011.
  
  \bibitem{MR2099200}
  R.~de~la~Bret\`eche and \'E.~Fouvry.
  \newblock L'\'{e}clat\'{e} du plan projectif en quatre points dont deux
    conjugu\'{e}s.
  \newblock {\em J. Reine Angew. Math.}, 576:63--122, 2004.
  
  \bibitem{MR1906155}
  A.~Chambert-Loir and Y.~Tschinkel.
  \newblock On the distribution of points of bounded height on equivariant
    compactifications of vector groups.
  \newblock {\em Invent. Math.}, 148(2):421--452, 2002.
  
  \bibitem{MR2740045}
  A.~Chambert-Loir and Y.~Tschinkel.
  \newblock Igusa integrals and volume asymptotics in analytic and adelic
    geometry.
  \newblock {\em Confluentes Math.}, 2(3):351--429, 2010.
  
  \bibitem{arXiv:1006.3345}
  A.~Chambert-Loir and Y.~Tschinkel.
  \newblock Integral points of bounded height on toric varieties.
  \newblock arXiv:1006.3345, 2010.
  
  \bibitem{MR2999313}
  A.~Chambert-Loir and Y.~Tschinkel.
  \newblock Integral points of bounded height on partial equivariant
    compactifications of vector groups.
  \newblock {\em Duke Math. J.}, 161(15):2799--2836, 2012.
  
  \bibitem{MR899402}
  J.-L. Colliot-Th\'el\`ene and J.-J. Sansuc.
  \newblock La descente sur les vari\'et\'es rationnelles. {II}.
  \newblock {\em Duke Math. J.}, 54(2):375--492, 1987.
  
  \bibitem{MR2520770}
  U.~Derenthal.
  \newblock Counting integral points on universal torsors.
  \newblock {\em Int. Math. Res. Not. IMRN}, (14):2648--2699, 2009.
  
  \bibitem{MR3269462}
  U.~Derenthal and C.~Frei.
  \newblock Counting imaginary quadratic points via universal torsors.
  \newblock {\em Compos. Math.}, 150(10):1631--1678, 2014.
  
  \bibitem{MR3348473}
  U.~Derenthal, J.~Hausen, A.~Heim, S.~Keicher, and A.~Laface.
  \newblock Cox rings of cubic surfaces and {F}ano threefolds.
  \newblock {\em J. Algebra}, 436:228--276, 2015.
  
  \bibitem{MR1230289}
  W.~Duke, Z.~Rudnick, and P.~Sarnak.
  \newblock Density of integer points on affine homogeneous varieties.
  \newblock {\em Duke Math. J.}, 71(1):143--179, 1993.
  
  \bibitem{MR1230290}
  A.~Eskin and C.~McMullen.
  \newblock Mixing, counting, and equidistribution in {L}ie groups.
  \newblock {\em Duke Math. J.}, 71(1):181--209, 1993.
  
  \bibitem{MR1381987}
  A.~Eskin, S.~Mozes, and N.~Shah.
  \newblock Unipotent flows and counting lattice points on homogeneous varieties.
  \newblock {\em Ann. of Math. (2)}, 143(2):253--299, 1996.
  
  \bibitem{MR974910}
  J.~Franke, Y.~I. Manin, and Y.~Tschinkel.
  \newblock Rational points of bounded height on {F}ano varieties.
  \newblock {\em Invent. Math.}, 95(2):421--435, 1989.
  
  \bibitem{MR3552013}
  C.~Frei and M.~Pieropan.
  \newblock O-minimality on twisted universal torsors and {M}anin's conjecture
    over number fields.
  \newblock {\em Ann. Sci. \'Ec. Norm. Sup\'er. (4)}, 49(4):757--811, 2016.
  
  \bibitem{MR2488484}
  A.~Gorodnik, H.~Oh, and N.~Shah.
  \newblock Integral points on symmetric varieties and {S}atake
    compactifications.
  \newblock {\em Amer. J. Math.}, 131(1):1--57, 2009.
  
  \bibitem{MR1786494}
  Y.~Hu and S.~Keel.
  \newblock Mori dream spaces and {GIT}.
  \newblock {\em Michigan Math. J.}, 48:331--348, 2000.
  \newblock Dedicated to William Fulton on the occasion of his 60th birthday.
  
  \bibitem{MR4104377}
  Z.~Huang and P.~Montero.
  \newblock Fano threefolds as equivariant compactifications of the vector group.
  \newblock {\em Michigan Math. J.}, 69(2):341--368, 2020.
  
  \bibitem{MR463151}
  V.~A. Iskovskih.
  \newblock Fano threefolds. {I}.
  \newblock {\em Izv. Akad. Nauk SSSR Ser. Mat.}, 41(3):516--562, 717, 1977.
  
  \bibitem{MR1199203}
  Y.~I. Manin.
  \newblock Notes on the arithmetic of {F}ano threefolds.
  \newblock {\em Compositio Math.}, 85(1):37--55, 1993.
  
  \bibitem{MR2286635}
  F.~Maucourant.
  \newblock Homogeneous asymptotic limits of {H}aar measures of semisimple linear
    groups and their lattices.
  \newblock {\em Duke Math. J.}, 136(2):357--399, 2007.
  
  \bibitem{MR641971}
  S.~Mori and S.~Mukai.
  \newblock Classification of {F}ano {$3$}-folds with {$B_{2}\geq 2$}.
  \newblock {\em Manuscripta Math.}, 36(2):147--162, 1981/82.
  
  \bibitem{MR1340296}
  E.~Peyre.
  \newblock Hauteurs et mesures de {T}amagawa sur les vari\'et\'es de {F}ano.
  \newblock {\em Duke Math. J.}, 79(1):101--218, 1995.
  
  \bibitem{MR2019019}
  E.~Peyre.
  \newblock Points de hauteur born\'ee, topologie ad\'elique et mesures de
    {T}amagawa.
  \newblock {\em J. Th\'eor. Nombres Bordeaux}, 15(1):319--349, 2003.
  \newblock Les XXII\`emes Journ\'ees Arithmetiques (Lille, 2001).
  
  \bibitem{MR1679841}
  P.~Salberger.
  \newblock Tamagawa measures on universal torsors and points of bounded height
    on {F}ano varieties.
  \newblock {\em Ast\'erisque}, (251):91--258, 1998.
  \newblock Nombre et r\'epartition de points de hauteur born\'ee (Paris, 1996).
  
  \bibitem{MR3117310}
  R.~Takloo-Bighash and Y.~Tschinkel.
  \newblock Integral points of bounded height on compactifications of semi-simple
    groups.
  \newblock {\em Amer. J. Math.}, 135(5):1433--1448, 2013.
  
  \bibitem{MR3438314}
  D.~Wei and F.~Xu.
  \newblock Counting integral points in certain homogeneous spaces.
  \newblock {\em J. Algebra}, 448:350--398, 2016.
  
  \end{thebibliography}
\end{document}